\documentclass[letterpaper, 10 pt, conference]{IEEEtran}  

\IEEEoverridecommandlockouts                              

\usepackage{color}
\definecolor{orange}{rgb}{1,0.5,0}

\usepackage{graphicx} 

\usepackage[cmex10]{amsmath}
\usepackage{amsthm}
\newtheorem{proposition}{Proposition}
\newtheorem{assumption}{Assumption}
\usepackage{amssymb} 

\usepackage{amsfonts}

\usepackage{slashbox} 

\usepackage{framed}

\usepackage{ctable}
\usepackage[hang]{subfigure}

\title{\LARGE \bf A robust trajectory tracking controller for four-wheel skid-steering mobile robots}

\author{Jae-Yun Jun, Minh-Duc Hua, Fa\"{\i}z Benamar

\thanks{Jae-Yun Jun, Minh-Duc Hua and Fa\"{\i}z Benamar are with: }
\thanks{1. Sorbonne Universit\'{e}s, UPMC Univ Paris 06, UMR 7222,}
\thanks{$\; \; \,$   ISIR, F-75005, Paris, France}
\thanks{2. CNRS, UMR 7222, ISIR, F-75005, Paris, France}
\thanks{e-mail: jaeyunjk@gmail.com, hua@isir.upmc.fr, amar@isir.upmc.fr}
}

\begin{document}

\maketitle
\thispagestyle{empty}
\pagestyle{empty}

\begin{abstract}
A novel dynamic model-based trajectory tracking control law is proposed for a four-wheel differentially driven mobile robot using a backstepping technique that guarantees the Lyapunov stability.
The present work improves the work of \textit{Caracciolo et al.} \cite{caracciolo.1999.icra}, a dynamic feedback linearization approach, by reducing the number of required assumptions
and the number of state terms. We also thoroughly investigate on a gain tuning procedure which is often overlooked for nonlinear controllers.
Finally, the performance of the proposed controller is compared with the dynamic feedback linearization approach via simulation results which indicate that
our controller is robust even in the presence of measurement noise and control time delay.
\end{abstract}
\section{Introduction}

\label{sec:Introduction}

Skid-steering mobile robots are often used for traversing over uneven terrains because they are mechanically robust
due to the reduced number of degrees of freedom and the non-requirement of active steering mechanisms.
They steer by creating a differential of the forces generated from the actuators located along the two sides of the longitudinal axis of the robot \cite{wong.2008.book}.
This differential of forces generates a non-null lateral velocity causing in turn the effect of side skidding.

The task of following a desired path (or trajectory) by a skid-steering mobile robot involves controlling the amount of the differential of the forces generated from the two sides of the robot
and therefore the amount of skidding.
However, controlling the amount of skidding is not an easy task. When a skid-steering robot follows a curved path,
its heading is not parallel to the tangent of the curved path because it laterally skids.
The robot's instantaneous center of rotation (ICR) is not fixed as in the case of active steering mobile robots with ideal rolling,
but it may continuously change, 
and, in extreme cases, the ICR may be located beyond the dimension of the robot along the longitudinal axis causing the robot's motion instability.

In the past, there have been several works in estimating the location of the ICR while a four-wheel skid-steering mobile robot \cite{mandow.2007.iros}
or a tracked mobile robot \cite{martinez.2005.ijrr} make turns with the purpose to improve in controlling these types of robots.
However, estimating the location of the ICR is not straightforward because it depends on the robot's instantaneous lateral velocity and its instantaneous angular velocity.

Aware of this difficulty, \textit{Caracciolo et al.} \cite{caracciolo.1999.icra} proposed
a model-based nonlinear controller in the dynamic feedback linearization paradigm that accounts
for the fact that the ICR does not lie along the lateral axis of the robot's center of mass but at a certain fixed distance
from the robot's center of mass along its longitudinal axis.
This notion is translated into an operational nonholonomic constraint and is added to the equations of motion
in order to ``virtually'' impose that the robot's lateral velocity must be proportional to its angular velocity \cite{caracciolo.1999.icra}.
However, their controller requires that the longitudinal velocity does not vanish at any time instant in order to have a finite-valued control input signal.
This implies that the robot should not have a non-null initial velocity, like the simulation example reported by the authors in \cite{caracciolo.1999.icra}.
Besides, their control input signal requires the measurement of the acceleration term in addition to the position and the velocity terms.

Recently, \textit{Koz\l{}owski and Paderski} proposed a controller to obtain practical stabilization in trajectory tracking.
This controller can stabilize the trajectory tracking up to certain bounds of the position and orientation errors \cite{kozlowski.2004.ijamcs}.
However, high gains are required in order to obtain sufficiently small error tracking,
bearing in mind that high gains may excessively amplify the destabilizing effects of measurement noise, control discretization and/or time delay.

In the present work, on top of the dynamic modeling that \textit{Caracciolo et al.} developed in \cite{caracciolo.1999.icra},
we propose a nonlinear control design that preserves the dynamics of the system which does not require nonzero-velocity constraint.
In addition, the acceleration term is not necessary, but only the position and the velocity terms suffice to control the robot.
For simulation illustration purposes, we emulate the sensory noise by adding a multi-variate white Gaussian noise to the state vector
and show that the proposed controller can robustify the system and track tightly a trajectory with the curvature changing continuously (an eight-shaped Lissajous curve).
Then, we further introduce to the noisy system a control time delay and a zero-order-hold to hold the control input signal during a certain period of time.
Reported simulation results show that the proposed controller is able to robustly track an eight-shaped Lissajous curve whereas the controller
based on the dynamic feedback linearization fails to track the reference trajectory.

The remainder of the paper is organized as follows. In Section II, the dynamic modeling of a four-wheel skid-steering mobile robot is recalled and discussed.
In Section III, we present the design of a novel controller and investigate on a gain-tuning procedure.
In Section IV, simulation results are reported and discussed.
Finally, conclusion remarks and perspectives are given in Section V.

\section{Recall on the dynamic modeling of a four-wheel skid-steering mobile robot with an operational nonholonomic constraint}
\label{sec: dynamical model}

\begin{figure} [!t]
 \centering
 \includegraphics[height=.35\textwidth, angle=0]{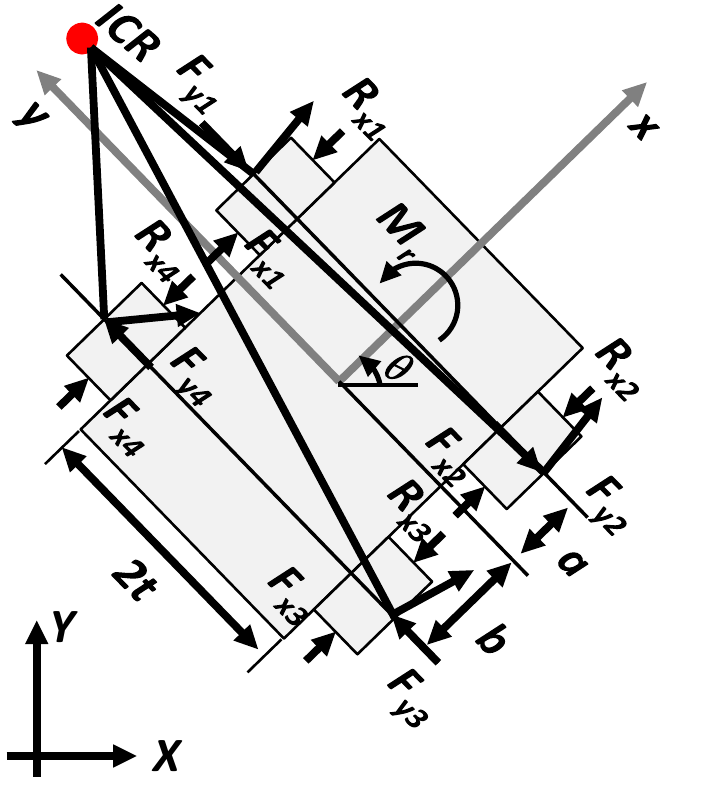} \vspace{-.05 in}
 \caption{A four-wheel skid-steering mobile robot.} \vspace{-.2 in}
 \label{fig:kinematic model}
\end{figure}

As shown in \cite{caracciolo.1999.icra}, the equations of motion of a four-wheel skid-steering mobile robot are given by
\begin{equation} \label{eq:equations of motion}
 \mathbf{M \ddot{q}} + \mathbf{c(q,\dot{q})} = \mathbf{E(q) \boldsymbol{\tau}},
\end{equation}
where $\mathbf{q} \triangleq \left[ X, Y, \theta \right]^T$  with $(X,Y)$ the coordinates of the robot's center of mass expressed in the inertial frame and $\theta$ its yaw angle. The terms $\mathbf{M}$, $\mathbf{c}$ and $\mathbf{E}$ and
the control vector $\boldsymbol{\tau}$ are defined as
\begin{equation*}\label{eq:definition of EoM's components}
\begin{split}
\mathbf{M} \triangleq
\begin{bmatrix}
 m & 0 & 0 \\
 0 & m & 0 \\
 0 & 0 & I
\end{bmatrix}, \quad
\mathbf{c(q,\dot{q})} \triangleq
\begin{bmatrix}
 R_x \cos{\theta} - F_y \sin{\theta} \\
 R_x \sin{\theta} + F_y \cos{\theta} \\
 M_r
\end{bmatrix}, \\
\mathbf{E(q)} \triangleq
\begin{bmatrix}
 \cos{\theta}/r & \cos{\theta}/r \\
 \sin{\theta}/r & \sin{\theta}/r \\
 t/r & -t/r
\end{bmatrix}, \quad
\tau_i = 2 r F_{x_i}, \quad i=1,2
\end{split}
\end{equation*}
where $m$, $I$ and $r$ denote the robot's mass, its inertia about the body $z$-axis and the wheel radius, respectively.
$a$, $b$ and $t$ are the robot's dimensional parameters (as defined in Fig. \ref{fig:kinematic model}).
$F_{x_i}$ is the $i$-th wheel's tractive force.
$R_x$, $F_y$ and $M_r$ are the resistive longitudinal and lateral forces and the resistive moment, respectively, which can be computed as follows
\begin{equation} \label{eq:total longitudinal resistive force}
\nonumber
 \begin{aligned}
    R_x =& \displaystyle\sum\limits_{i=1}^4 R_{xi} = f_r \frac{mg}{2} \left( \text{sgn}(\dot{x}_1) + \text{sgn}(\dot{x}_2) \right), \\
    F_y =& \displaystyle\sum\limits_{i=1}^4 F_{yi} = \mu \frac{mg}{a+b} \left( b \: \text{sgn}(\dot{y}_1) + a \: \text{sgn}(\dot{y}_3) \right), \\
    M_r =&  \quad a ( F_{y1} + F_{y2} ) - b ( F_{y3} + F_{y4} ) \\
            & + t \left[ (R_{x2}+R_{x3}) - (R_{x1} + R_{x4}) \right] \\
        =&   \quad \mu \frac{a \, b \, m \, g}{a+b} \left( \text{sgn}(\dot{y}_1) - \text{sgn}(\dot{y}_3) \right) \\
            & + f_r \frac{t \, m \, g}{2} \left( \text{sgn}(\dot{x}_2) - \text{sgn}(\dot{x}_1) \right),
 \end{aligned}
\end{equation}
with $g$, $f_r$, $\mu$, and $\text{sgn}(\cdot)$ the gravitational acceleration, the coefficient of rolling friction, the coefficient of lateral friction and the sign function, respectively. Besides, $\dot{x}_i$ and $\dot{y}_i$, with $i=1, \cdots,4$, are respectively the longitudinal and the lateral wheel velocities, subject to the following relationships with the linear and angular velocities $(\dot{x}, \dot y, \dot{\theta})$ expressed in the body frame
\begin{equation*} \label{eq:relation between wheel velocities and CoM velocities}
 \left\{
 \begin{split}
  \dot{x}_1 & = \dot{x}_4 = \dot{x} - t \dot{\theta}, \\ 
  \dot{x}_2 & = \dot{x}_3 = \dot{x} + t \dot{\theta}, \\
  \dot{y}_1 & = \dot{y}_2 = \dot{y} + a \dot{\theta}, \\
  \dot{y}_3 & = \dot{y}_4 = \dot{y} - b \dot{\theta}.
 \end{split}
 \right.
\end{equation*}

The velocity in the body frame is related to the velocity in the inertial frame as follows
\begin{equation}
 \begin{bmatrix}
  \dot{X} \\ \dot{Y}
 \end{bmatrix}
 =
 \mathbf{R}
 \begin{bmatrix}
  \dot{x} \\ \dot{y}
 \end{bmatrix},
\end{equation}
with $\mathbf{R}  \triangleq
 \begin{bmatrix}
  \cos{\theta} & -\sin{\theta} \\
  \sin{\theta} & \cos{\theta}
 \end{bmatrix}$  the  rotation matrix.

The location $(x_{\text{ICR}}, y_{\text{ICR}})$, expressed in the body frame, of the instantaneous center of rotation (ICR) should remain inside the robot's dimension
along the longitudinal direction (i.e., $-b \leq x_{\text{ICR}}  \leq a$) at any time instant in order to ensure the robot's motion stability.
If the ICR goes outside the robot's dimension along the longitudinal direction, then all the resistive lateral forces $F_{yi}$, with $i=1,\cdots, 4$,
will have the same sign, and, consequently, there is no way to balance the amount of skidding with the wheel actuators, causing the loss of controllability of the mobile robot \cite{shiller.1993.icra}.
If the location of the ICR is known, then a controller may be designed to track a reference trajectory while ensuring the constraint $-b \leq x_{\text{ICR}}  \leq a$ so as to avoid instability.
However, it is not easy to design such a controller due to the fact that the longitudinal coordinate of ICR, $x_{\text{ICR}} = - {\dot{y}}/{\dot{\theta}}$ (see \cite{caracciolo.1999.icra}),
is a function of the vehicle's state.
A practical solution has been proposed by \textit{Caracciolo et al.} \cite{caracciolo.1999.icra} by imposing a ``virtual'' constraint $x_{\text{ICR}} = d_0$,
with $0 < d_0 < a$. This yields the following nonholonomic constraint
\begin{equation} \label{eq:operational nonholonomic constraint}
 \dot{y} + d_0 \dot{\theta}  = 0,
\end{equation}
which implies that the lateral speed and the angular velocity should have a fixed relationship by the constant distance $d_0$.
This ``unnatural'' constraint is not always satisfied in reality, and a controller should be designed to closely maintain this relationship.
Inspired by \cite{caracciolo.1999.icra}, the control design in the next section is based on the following augmented model (instead of \eqref{eq:equations of motion})
\begin{equation} \label{eq:equations of constrained motion}
 \mathbf{M \ddot{q}} + \mathbf{c(q,\dot{q})} = \mathbf{E(q)}\boldsymbol{\tau} + \mathbf{A(q)}^\mathrm{T} \boldsymbol{\lambda},
\end{equation}
where $\boldsymbol{\lambda}$ is a vector of Lagrangian multipliers representing the constrained forces, while the matrix $\mathbf{A}$ holds the following relationship
\begin{equation*} \label{eq:operational nonholonomic constraint 1}
 \begin{bmatrix}
  -\sin{\theta} & \cos{\theta} & d_0
 \end{bmatrix}
 \begin{bmatrix}
  \dot{X} \\ \dot{Y} \\ \dot{\theta}
 \end{bmatrix}
 = \mathbf{A(q) \dot{q}} = 0.
\end{equation*}
The admissible generalized velocities $\dot{\mathbf{q}}$ can be defined as
\begin{equation} \label{eq:admissible generalized velocities}
 \mathbf{\dot{q}} = \mathbf{N(q)} \boldsymbol{\eta},
\end{equation}
where $\boldsymbol{\eta}  \in \mathbb{R}^2$ is a pseudo-velocity, and the columns of the matrix $\mathbf{N}$ are in the null space of $\mathbf{A}$, e.g.,
\begin{equation*} \label{eq:null space of A}
 \mathbf{N(q)} =
 \begin{bmatrix}
  \cos{\theta} & -\sin{\theta} \\
  \sin{\theta} & \cos{\theta} \\
  0 & -\frac{1}{d_0}
 \end{bmatrix}.
\end{equation*}
By differentiating \eqref{eq:admissible generalized velocities} and eliminating $\boldsymbol{\lambda}$ from \eqref{eq:equations of constrained motion} one obtains
\begin{equation} \label{eq:first simplified EoMs}
\left\{
 \begin{aligned}
  \mathbf{\dot{q}} & = \mathbf{N} \boldsymbol{\eta}, \\
  \mathbf{N}^\mathrm{T}\mathbf{M N} \mathbf{\dot{\boldsymbol{\eta}}} & =  \mathbf{N}^\mathrm{T} ( \mathbf{E} \boldsymbol{\tau} - \mathbf{M \dot{N}} \boldsymbol{\eta} - \mathbf{c}).
 \end{aligned}
\right.
\end{equation}
One verifies that the matrices $\mathbf{N}^\mathrm{T}\mathbf{M N}$ and $ \mathbf{N}^\mathrm{T}\mathbf{E}$ are invertible. Thus, by making simple change of control variables
\begin{equation} \label{eq:static state feedback law}
\boldsymbol{\tau} = \left(\mathbf{N}^\mathrm{T}\mathbf{E}\right)^{-1} \left( \mathbf{N}^\mathrm{T}\mathbf{M N} \mathbf{u} +\mathbf{N}^\mathrm{T} \mathbf{M \dot{N} \boldsymbol{\eta}} + \mathbf{N}^\mathrm{T}\mathbf{c} \right),
\end{equation}
with $\mathbf{u} = \begin{bmatrix} u_1 & u_2 \end{bmatrix}^\mathrm{T}$ the vector of new control variables, then system \eqref{eq:first simplified EoMs} can be rewritten as
\begin{equation} \label{eq:simplified equations of motion}
\left\{
 \begin{aligned}
  \mathbf{\dot{q}} & = \mathbf{N} \boldsymbol{\eta}, \\
  \mathbf{\dot{\boldsymbol{\eta}}} & = \mathbf{u},
 \end{aligned}
\right.
\end{equation}
which is equivalent to
\begin{equation} \label{eq:explicit form of EoMs}
 \left\{
 \begin{split}
  \dot{X} & = \cos{\theta} \eta_1 - \sin{\theta} \eta_2, \\
  \dot{Y} & = \sin{\theta} \eta_1 + \cos{\theta} \eta_2, \\
  \dot{\theta} & = -\frac{1}{d_0} \eta_2, \\
  \dot{\eta}_1 & = u_1, \\
  \dot{\eta}_2 & = u_2.
 \end{split}
 \right.
\end{equation}
Th control design proposed in the next section is based on (\ref{eq:explicit form of EoMs}).

\section{Lyapunov-based control design}
\label{sec:lyapunov-based controller}

\subsection{Control design}
\label{subsec:controller design}

Similar to \cite{caracciolo.1999.icra}, a control point is chosen on the longitudinal body axis at the distance $d_0$ from the origin of the body frame. The vector of coordinates expressed in the inertial frame of this control point is thus given by
\begin{equation*} \label{eq:control point coordinates}
 \boldsymbol{\xi}  =
 \begin{bmatrix} X + d_0 \cos{\theta} \\ Y + d_0 \sin{\theta} \end{bmatrix}.
\end{equation*}
From \eqref{eq:explicit form of EoMs}, one verifies that the time-derivative of $\boldsymbol{\xi}$ satisfies
\[
\dot{\boldsymbol{\xi}} = \eta_1 \mathbf{R} \mathbf{e}_1, \quad \text{with} \quad \mathbf{e}_1 \triangleq \begin{bmatrix} 1 & 0 \end{bmatrix}^\mathrm{T}.
\]
Let $\boldsymbol{\xi}_r \in \mathbb{R}^2$ denote the reference position expressed in the inertial frame for the control point defined up to third-order derivative. Define $\tilde {\boldsymbol{\xi}} \triangleq \boldsymbol{\xi} - \boldsymbol{\xi}_r$ and $\boldsymbol{\bar{\xi}} \triangleq \mathbf{R}^\mathrm{T} \boldsymbol{\tilde{\xi}}$ as the position errors expressed in the inertial frame and body frame, respectively.

It is straightforward to deduce following equations of the error dynamics
\begin{equation} \label{eq:new system of equations}
 \left\{
 \begin{aligned}
  \boldsymbol{\dot{\bar{{\xi}}}} & =  -\omega \mathbf{S} \boldsymbol{\bar{\xi}} + \eta_1 \mathbf{e_1} - \mathbf{R}^\mathrm{T} \boldsymbol{\dot{\xi}}_r, \\
  \dot{\mathbf{R}} & = \omega \mathbf{R} \mathbf{S}, \\
  \dot{\eta}_1 & = u_1, \\
  \dot{\omega} & = \bar{u}_2,
 \end{aligned}
 \right.
\end{equation}
with  $\bar{u}_2 \triangleq -\frac{1}{d_0} u_2$ the new control variable, $\omega \triangleq \dot{\theta}$ and $\mathbf{S} \triangleq \begin{bmatrix} 0 & -1 \\ 1 & 0 \end{bmatrix}$.
Then, the control objective can be stated as the asymptotical stabilization of $\bar {\boldsymbol{\xi}}$,
or equivalently of $\tilde {\boldsymbol{\xi}}$, about zero using $(u_1, \bar u_2)$ as control inputs.

The first equation of \eqref{eq:new system of equations} indicates that the relation $\boldsymbol{\bar{\xi}} \equiv \mathbf{0}$ implies that
\begin{equation}\label{eq:equilibrium}
\eta_1 \mathbf{e_1} - \mathbf{R}^\mathrm{T} \boldsymbol{\dot{\xi}}_r \equiv \mathbf{0}.
\end{equation}
As long as $|\boldsymbol{\dot{\xi}}_r|$ is different from zero, one can define a locally unique solution of $\mathbf{R}$ (or $\theta$) to equation \eqref{eq:equilibrium}.
However, this solution cannot be prolonged by continuity at $\boldsymbol{\dot{\xi}}_r=\mathbf{0}$.
This singularity corresponds to the case when the linearization of system
\eqref{eq:new system of equations} at any equilibrium $(\boldsymbol{\bar{\xi}}, \mathbf{R}, \omega) = (\mathbf{0}, \mathbf{R}^*, 0)$ is not controllable.
Moreover, one can verify from the application of the Brockett's theorem \cite{brockett.1983.dgct} for this case that there does not exist any time-invariant $\mathcal{C}^1$ feedback control law
that asymptotically stabilizes the system at the equilibrium $(\boldsymbol{\bar{\xi}}, \mathbf{R}, \omega) = (\mathbf{0}, \mathbf{R}^*, 0)$.
We thus discard this difficult issue in the present paper by making the following reasonable assumption.
\begin{assumption}\label{assump1}
There exists a positive constants $\delta_r$ and $a_r$ such that $|\boldsymbol{\dot{\xi}}_r(t)| \geq \delta_r$ and $|\boldsymbol{\ddot{\xi}}_r(t)|\leq a_r$, $\forall t$.
\end{assumption}

The following result is obtained based on a Lyapunov function constructed using a backstepping proceduce.

\begin{proposition} \label{propos:asymptotic convergence}
Consider the error system \eqref{eq:new system of equations}.
Assume that Assumption \ref{assump1} holds. Let $\eta_{1d}$ and $\omega_d$ denote auxiliary control variables derived from the backstepping procedure and defined as
\begin{equation} \label{eq: designed velocities}
 \left\{
 \begin{aligned}
  \eta_{1d} & \triangleq \mathbf{e}_1^\mathrm{T} \mathbf{R}^\mathrm{T} \boldsymbol{\dot{\xi}}_r - k_1 \bar{\xi}_1, \\
  \omega_d  & \triangleq \omega_r - k_2 |\boldsymbol{\dot{\xi}}_r| \bar{\xi}_2 + k_3 |\boldsymbol{\dot{\xi}}_r| \left( \mathbf{e}_2^\mathrm{T} \mathbf{R}^\mathrm{T} \mathbf{\dot{\xi}_r} \right),
 \end{aligned}
 \right.
\end{equation}
where $\omega_r \triangleq - \frac{\mathbf{\dot{\xi}_r^T} \mathbf{S} \mathbf{\ddot{\xi}_r}}{|\mathbf{\dot{\xi}_r}|^2}$, $k_{1,2}$ are positive constant gains,
and $k_3$ is a positive gain (not necessarily constant) satisfying $\inf\limits_{t} k_3(t) > 0$. Apply the following control law
\begin{equation} \label{eq:control input signals}
 \left\{
 \begin{split}
  u_1 & = \dot{\eta}_{1d} - k_4 \bar{\xi}_1 - k_6 (\eta_1 - \eta_{1d}), \\
  \bar{u}_2 & = \dot{\omega}_d + \frac{k_5}{k_2} \frac{\mathbf{e}_2^\mathrm{T} \mathbf{R}^\mathrm{T} \mathbf{\dot{\xi}_r}}{|\mathbf{\dot{\xi}_r}|},
                 + \left( \frac{\dot{k}_5}{2 k_5} - k_7 \right) (\omega - \omega_d),
 \end{split}
 \right.
\end{equation}
where $k_{4,6}$ are positive constant gains and $k_{5,7}$ are positive gains (not necessarily constant) satisfying $\inf\limits_{t} k_{5,7}(t) > 0$. Then, the following properties hold:
\begin{enumerate}
\item There exist only two equilibria $(\boldsymbol{\bar{\xi}}, \mathbf{R}, \omega) = (\mathbf{0}, \mathbf{R}_{\pm}^*, 0)$, with $\mathbf{R}_{+}^{*\mathrm{T}} \mathbf{e}_1 =  \frac{\boldsymbol{\dot{\xi}}_r}{|\boldsymbol{\dot{\xi}}_r|}$ and $\mathbf{R}_{-}^{*\mathrm{T}} \mathbf{e}_1 =  -\frac{\boldsymbol{\dot{\xi}}_r}{|\boldsymbol{\dot{\xi}}_r|}$.
\item The equilibrium $(\boldsymbol{\bar{\xi}}, \mathbf{R}, \omega) = (\mathbf{0}, \mathbf{R}_{+}^*, \omega_r)$ is almost-globally asymptotically stable.
\end{enumerate}
\end{proposition}
\begin{proof}
The first property of Proposition \ref{propos:asymptotic convergence} can be straightforwardly deduced from \eqref{eq:new system of equations} and \eqref{eq:equilibrium}.
We prove now the second property.
Consider the following storage function
\begin{equation} \label{eq:storage function}
{\cal S} \triangleq \frac{1}{2} |\boldsymbol{\bar{\xi}}|^2 + \frac{1}{k_2}\left(1 - \frac{\mathbf{e}_1^\mathrm{T} \mathbf{R}^\mathrm{T}\boldsymbol{\dot{\xi}}_r}{|\boldsymbol{\dot{\xi}}_r|}\right),\quad (k_2>0)
\end{equation}
whose time-derivative along the system's solutions satisfies (using Lemma 5 in \cite{hua.2009.phdthesis})
\[
\begin{split}
\dot{\cal S} &= \boldsymbol{\bar{\xi}}^\mathrm{T}\left(\eta_1 \mathbf{e_1} - \mathbf{R}^\mathrm{T} \boldsymbol{\dot{\xi}}_r \right) \\
& \quad+ \frac{1}{k_2} \left( \frac{\omega \mathbf{e}_1^\mathrm{T} \mathbf{S} \mathbf{R}^\mathrm{T}\boldsymbol{\dot{\xi}}_r}{|\boldsymbol{\dot{\xi}}_r|}
-\mathbf{e}_1^\mathrm{T}\mathbf{R}^\mathrm{T} \frac{d}{dt} \left(\frac{\boldsymbol{\dot{\xi}}_r}{|\boldsymbol{\dot{\xi}}_r|} \right)  \right)\\
&= \bar{\xi}_1 \left( \eta_1 - \mathbf{e}_1^\mathrm{T} \mathbf{R}^\mathrm{T}\boldsymbol{\dot{\xi}}_r \right)
- \frac{\mathbf{e}_2^\mathrm{T} \mathbf{R}^\mathrm{T}\boldsymbol{\dot{\xi}}_r}{k_2 |\mathbf{\dot{\xi}_r}|} \left( \omega - \omega_r + k_2 |\mathbf{\dot{\xi}_r}| \bar{\xi}_2 \right),
\end{split}
\]
with $\omega_r$ defined in Proposition \ref{propos:asymptotic convergence}. Then, using the expressions \eqref{eq: designed velocities} of the auxiliary control variables $\eta_{1d}$ and $\omega_d$ one deduces
\begin{equation*} \label{eq:simplified derivative combined container}
\begin{split}
\dot{\cal S} \!=\! -\! k_1 \bar{\xi}_1^{\,2} \!-\! \frac{k_3}{k_2} (\mathbf{e}_2^\mathrm{T} \mathbf{R}^\mathrm{T} \mathbf{\dot{\xi}_r})^2 \!\!+\!  \bar{\xi}_1(\eta_1 \!\!-\! \eta_{1d}) \!-\! \frac{\mathbf{e}_2^\mathrm{T} \mathbf{R}^\mathrm{T} \boldsymbol{\dot{\xi}_r}}{|\mathbf{\dot{\xi}_r}|} (\omega \!-\! \omega_d).
\end{split}
\end{equation*}
Now, backstepping procedure can be applied to deduce the real control inputs $(u_1,\bar u_2)$. Consider the following Lyapunov candidate function
\begin{equation} \label{eq:lyapunov function}
{\cal L} \triangleq {\cal S} + \frac{1}{2k_4} \left( \eta_1 - \eta_{1d} \right)^2 + \frac{1}{2k_5} \left( \omega - \omega_d \right)^2,
\end{equation}
with $\cal S$ defined by \eqref{eq:storage function}. From the system \eqref{eq:new system of equations} and the control expressions \eqref{eq:control input signals}, one deduces
\begin{equation} \label{eq:derivative lyapunov equation}
\begin{split}
\!\!\dot{{\cal L}} &= \dot{{\cal S}} + \frac{1}{k_4} \left( \eta_1 - \eta_{1d} \right) \left( u_1 - \dot{\eta}_{1d} \right) \\
&\qquad + \frac{1}{k_5} (\omega - \omega_d) (\bar{u}_2 - \dot{\omega}_d) - \frac{\dot{k}_5}{2 k_5^2}(\omega-\omega_d)^2 \\
&\!\!\!\!\!\!\!\!= - k_1 \bar{\xi}_1^{\,2} \!-\! \frac{k_3}{k_2} (\mathbf{e}_2^\mathrm{T} \mathbf{R}^\mathrm{T} \boldsymbol{\dot{\xi}}_r)^2 \!-\! \frac{k_6}{k_4} \left( \eta_1 \!-\! \eta_{1d} \right)^2 \!-\! \frac{k_7}{k_5} (\omega \!-\! \omega_d)^2 \!\!\!\!.
\end{split}
\end{equation}
Since $\dot{\cal L}$ is negative semi-definite, the terms $\boldsymbol{\bar{\xi}}$, $\eta_1 - \eta_{1d}$ and $\omega - \omega_d$ remain bounded. From the boundedness of the reference acceleration $\boldsymbol{\ddot{\xi}_r}$ (Assumption \ref{assump1}),
one can show that $\ddot{\cal L}$ is bounded which implies that $\dot{\cal L}$ is uniformly continuous along every system's solution. Then, by the application of the Barbalat's lemma \cite{khalil.1996.book},
one can ensure that $\dot{\cal L}$ converges to zero. Consequently, one can deduce that
\begin{equation}\label{eq:aysmptotic convergence}
\left( \bar{\xi}_1, \mathbf{e}_2^\mathrm{T} \mathbf{R}^\mathrm{T} \boldsymbol{\dot{\xi}}_r, \eta_1 - \eta_{1d}, \omega - \omega_d \right) \rightarrow \mathbf{0}.
\end{equation}

In addition, one needs to make sure that $\bar{\xi}_2$ asymptotically converges to zero.
If $u_1$ and $\bar{u}_2$ are defined as (\ref{eq:control input signals}), then $\omega$ converges to $\omega_{d}$ as indicated in (\ref{eq:aysmptotic convergence}).
Using this fact and the Lemma 5 of \cite{hua.2009.phdthesis}, one gets
\begin{equation} \label{eq:proof 1}
 \frac{d}{dt} \left( \frac{\mathbf{e}_2^\mathrm{T} \mathbf{R}^\mathrm{T} \boldsymbol{\dot{\xi}}_r}{|\boldsymbol{\dot{\xi}}_r|} \right) \rightarrow
 - \frac{\mathbf{e}_1^\mathrm{T} \mathbf{R}^\mathrm{T} \boldsymbol{\dot{\xi}}_r}{|\boldsymbol{\dot{\xi}}_r|} \left( \omega_d  - \omega_r \right).
\end{equation}
From (\ref{eq:aysmptotic convergence}), $\mathbf{e}_2^\mathrm{T} \mathbf{R}^\mathrm{T} \boldsymbol{\dot{\xi}}_r$ converges to zero. Therefore, the $\omega_d$ given in (\ref{eq: designed velocities}) converges to
\begin{equation} \label{eq: new designed angular velocity}
\omega_d \rightarrow\omega_r - k_2 |\mathbf{\dot{\xi}_r}| \bar{\xi}_2.
\end{equation}
Using (\ref{eq: new designed angular velocity}) in (\ref{eq:proof 1}), one gets
\begin{equation} \label{eq:proof 2}
 \frac{d}{dt} \left( \frac{\mathbf{e}_2^\mathrm{T} \mathbf{R}^\mathrm{T} \boldsymbol{\dot{\xi}}_r}{|\boldsymbol{\dot{\xi}}_r|} \right) \rightarrow
 k_2  (\mathbf{e}_1^\mathrm{T} \mathbf{R}^\mathrm{T} \boldsymbol{\dot{\xi}}_r) \bar{\xi}_2.
\end{equation}
On the other hand, since $\left( \mathbf{e_2^T} \mathbf{R^T} \mathbf{\dot{\xi}_r} \right) \rightarrow 0$ holds (from (\ref{eq:aysmptotic convergence})),
$\frac{d}{dt}\left( \mathbf{e_2^T} \mathbf{R^T} \mathbf{\dot{\xi}_r} \right) \rightarrow 0$ must be true.
Using Assumption \ref{assump1}, $|\mathbf{\dot{\xi}_r}| \neq 0$.
Hence,  $\frac{d}{dt}\left( \frac{\mathbf{e_2^T} \mathbf{R^T} \mathbf{\dot{\xi}_r}}{|\mathbf{\dot{\xi}_r}|} \right) \rightarrow 0$ must also be true.
Therefore, $\mathbf{e_1^T R^T \dot{\xi}_r} k_2 \bar{\xi}_2 \rightarrow 0$. But, $\mathbf{e_1^T R^T \dot{\xi}_r} \nrightarrow 0$, and $k_2 >0$. Therefore, $\bar{\xi}_2 \rightarrow 0$.
\end{proof}

\subsection{Gain tuning}
\label{subsec: gain tuning}
Generally, gain tuning for nonlinear control laws is less obvious than for linear control ones. However, we will show how the gains for our proposed controller can be tuned by using existing tuning techniques in linear control theory. A simple way to determine the control gains consists in using the pole placement technique for the linearization of the system \eqref{eq:new system of equations} at the equilibrium and for a particular reference trajectory such as a straight line or a circle with constant forward speed. In this case, one deduces that
\begin{equation} \label{eq:derivative error coordinates in body frame simplified}
 \begin{split}
  \boldsymbol{\dot{\bar{\xi}}} 
                           \approx & \; \eta_1 \mathbf{e_1} - \mathbf{R}^\mathrm{T} \boldsymbol{\dot{\xi}}_r
   =
  \begin{bmatrix}
   \eta_1 - \mathbf{e}_1^\mathrm{T} \mathbf{R}^\mathrm{T} \boldsymbol{\dot{\xi}}_r \\
   - \mathbf{e}_2^\mathrm{T} \mathbf{R}^\mathrm{T} \boldsymbol{\dot{\xi}}_r
  \end{bmatrix}.
 \end{split}
\end{equation}
Then, defining $\tilde{\eta}_1 \triangleq \eta_1 - \eta_{1d}$ and using the definition of $\eta_{1d}$ given in (\ref{eq: designed velocities}), one obtains
\begin{equation} \label{eq:linearized derivative error in body frame}
 \boldsymbol{\dot{\bar{\xi}}} = \begin{bmatrix} \dot{\bar{\xi}}_1 \\ \dot{\bar{\xi}}_2 \end{bmatrix} \approx
 \begin{bmatrix}
  \eta_{1d} + \tilde{\eta}_1 - \mathbf{e}_1^\mathrm{T} \mathbf{R}^\mathrm{T} \boldsymbol{\dot{\xi}}_r \\
  - \mathbf{e}_2^\mathrm{T} \mathbf{R}^\mathrm{T} \boldsymbol{\dot{\xi}}_r
 \end{bmatrix}
 =
 \begin{bmatrix}
    -k_1 \bar{\xi}_1 + \tilde{\eta}_1 \\
    -\mathbf{e}_2^\mathrm{T} \mathbf{R}^\mathrm{T} \boldsymbol{\dot{\xi}}_r
 \end{bmatrix}.
\end{equation}
On the other hand, by differentiating $\tilde \eta_1$ and by using (\ref{eq:control input signals}), one deduces
\begin{equation} \label{eq:derivative error in eta1}
 \dot{\tilde{\eta}}_1 = \dot{\eta}_1 - \dot{\eta}_{1d} = u_1 - \dot{\eta}_{1d} = - k_4 \bar{\xi}_1 - k_6 (\eta_1 - \eta_{1d}).
\end{equation}
One can regroup the expressions for $\dot{\bar{\xi}}_1$ (from (\ref{eq:linearized derivative error in body frame})) and $\dot{\tilde{\eta}}_1$ (from (\ref{eq:derivative error in eta1})) in matrix form as follows
\begin{equation*} \label{eq:first gain analysis matrix}
 \begin{bmatrix}
    \dot{\bar{\xi}}_1 \\ \dot{\tilde{\eta}}_1
 \end{bmatrix}
 =
 \begin{bmatrix}
    -k_1 & 1 \\
    -k_4 & -k_6
 \end{bmatrix}
 \begin{bmatrix}
    \bar{\xi}_1 \\ \tilde{\eta}_1
 \end{bmatrix}
 =
 \mathbf{A_1}
 \begin{bmatrix}
    \bar{\xi}_1 \\ \tilde{\eta}_1
 \end{bmatrix}.
\end{equation*}
Now, the gains $k_1$, $k_4$ and $k_6$ can be chosen such that $\mathbf{A_1}$ is Hurwitz. One verifies that the characteristic polynomial of $\mathbf{A_1}$ given by
\[
P_1(\lambda) = \lambda^2 + (k_1 + k_6) \lambda + k_1 k_6 + k_4
\]
is Hurwitz if $(k_1+k_6) k_1k_6 > k_4$. For instance, given two negative real numbers $\lambda_{1,2}<0$ and choosing
\begin{equation}\label{eq:conditions for k1, k4, k6}
\left\{
\begin{split}
k_1 &< - \max(\lambda_1, \lambda_2), \\
k_6 &= -\lambda_1 -\lambda_2 -k_1, \\
k_4 &= (k_1 + \lambda_1) (k_1 + \lambda_2),
\end{split}
\right.
\end{equation}
one ensures the positivity of $k_1$, $k_4$ and $k_6$ and that the matrix $\mathbf{A_1}$ is Hurwitz with two negative real poles $\lambda_{1,2}<0$.

Now, let $\tilde{\theta}$ be the angle formed between $\mathbf{R e}_1$ and $\frac{\boldsymbol{\dot{\xi}_r}}{|\boldsymbol{\dot{\xi}}|}$ (i.e., $\cos{\tilde{\theta}} = \left( \mathbf{R e}_1 \right)^\mathrm{T} \frac{\boldsymbol{\dot{\xi}_r}}{|\boldsymbol{\dot{\xi}}_r|}$). Then, in the first order approximation, one has $\tilde{\theta} \approx - \frac{\mathbf{e}_2^\mathrm{T} \mathbf{R}^\mathrm{T}\boldsymbol{\dot{\xi}}_r}{|\boldsymbol{\dot{\xi}}_r|}$. One can easily verifies that
\begin{equation*} \label{eq:derivative error in yaw simplified}
 \dot{\tilde{\theta}} \approx \tilde{\omega} - k_2 |\mathbf{\dot{\xi}_r}| \bar{\xi}_2 - k_3 |\mathbf{\dot{\xi}_r}|^2 \tilde{\theta},
\end{equation*}
with $\tilde{\omega} \triangleq \omega - \omega_d$. Besides, by differentiating $\tilde{\omega}$ and using (\ref{eq:control input signals}) one also verifies that
\[
\dot{\tilde{\omega}} = \bar u_2 -\dot{\omega}_d = \frac{k_5}{k_2} \frac{\mathbf{e}_2^\mathrm{T} \mathbf{R}^\mathrm{T}\boldsymbol{\dot{\xi}}_r}{|\boldsymbol{\dot{\xi}}_r|}  - k_7 \, \tilde{\omega} = - \frac{k_5}{k_2} \tilde{\theta} - k_7 \, \tilde{\omega}.
\]
From here, one deduces the following second linearized subsystem in matrix form
\begin{equation*} \label{eq:second gain analysis matrix}
 \begin{bmatrix}
    \dot{\bar{\xi}}_2 \\ \dot{\tilde{\theta}} \\ \dot{\tilde{\omega}}
 \end{bmatrix}
 \!\!=\!\!
 \begin{bmatrix}
    0                           & |\boldsymbol{\dot{\xi}}_r|         & 0 \\
    -k_2 |\boldsymbol{\dot{\xi}}_r| & - k_3 |\boldsymbol{\dot{\xi}}_r|^2 & 1 \\
    0                           & -\frac{k_5}{k_2}               & -k_7
 \end{bmatrix}
 \begin{bmatrix}
    \bar{\xi}_2 \\ \tilde{\theta} \\ \tilde{\omega}
 \end{bmatrix}
 \! = \mathbf{A_2} \!\begin{bmatrix}
    \bar{\xi}_2 \\ \tilde{\theta} \\ \tilde{\omega}
 \end{bmatrix}.
\end{equation*}
It matters now to choose the gains $k_2$, $k_3$, $k_5$ and $k_7$ such that the characteristic polynomial of $\mathbf{A_2}$ given by
\begin{equation*} \label{eq:characteristic polynomial for A2}
\begin{split}
P_2(\lambda)
  &=  \lambda^3 + \lambda^2 (k_3 |\boldsymbol{\dot{\xi}}_r|^2 + k_7)  \\&\quad+ \lambda \left( k_3 k_7 |\boldsymbol{\dot{\xi}}_r|^2 + \frac{k_5}{k_2} + k_2 |\boldsymbol{\dot{\xi}}_r|^2 \right) + |\boldsymbol{\dot{\xi}}_r|^2 k_2 k_7
\end{split}
\end{equation*}
is Hurwitz. To simplify the task, let us set
\[
k_3 = \frac{\kappa_3}{|\boldsymbol{\dot{\xi}}_r|},\ k_5 = |\boldsymbol{\dot{\xi}}_r|^2 \kappa_5,\ k_7 = |\boldsymbol{\dot{\xi}}_r| \kappa_7,
\]
with $\kappa_3, \ \kappa_5,\ \kappa_7$ positive constants. Then, the polynomial $P_2(\lambda)$ can be factorized as
\begin{equation*} \label{eq:factorized characteristic polynomial for A2}
\begin{split}
P_2(\lambda) & = \lambda^3 + \lambda^2 |\boldsymbol{\dot{\xi}}_r| (\kappa_3  + \kappa_7) \\
&\quad + \lambda |\boldsymbol{\dot{\xi}}_r|^2 \left( \kappa_3 \kappa_7 + \frac{\kappa_5}{k_2} + k_2 \right) + |\boldsymbol{\dot{\xi}}_r|^3 k_2 \kappa_7.
\end{split}
\end{equation*}
From the above expression of $P_2(\lambda)$, one may set poles for this characteristic polynomial depending on the norm of the reference velocity as
$\lambda_{1,2,3} = |\boldsymbol{\dot{\xi}}_r|^2 \bar{\lambda}_{1,2,3}$, with $\bar{\lambda}_{1,2,3}$ negative real numbers. This implies the following relations
\begin{equation} \label{eq:conditions for k2, k3, k5, k7}
\left\{ \begin{split}
  \kappa_7 & = - \frac{\bar{\lambda}_1 \bar{\lambda}_2 \bar{\lambda}_3}{k_2},  \\
  \kappa_3 & = - \bar{\lambda}_1 -  \bar{\lambda}_2 - \bar{\lambda}_3 - \kappa_7, \\
  \kappa_5 & = k_2 \left( \bar{\lambda}_1 \bar{\lambda}_2 + \bar{\lambda}_1 \bar{\lambda}_3 + \bar{\lambda}_2 \bar{\lambda}_3 -\kappa_3 \kappa_7 - k_2 \right).
 \end{split} \right.
\end{equation}
Then, the values of $\bar\lambda_{1,2,3} \ (<0)$ and $k_2 \ (>0)$ should be chosen such that $\kappa_3$, $\kappa_5$ and $\kappa_7$ computed according to \eqref{eq:conditions for k2, k3, k5, k7} are positive. For instance, by setting $\bar{\lambda}_{1,2}$ equal, and choosing $k_2\ (>0)$ and $\bar\lambda_3\ (<0)$ such that
\begin{equation} \label{eq:example conditions for k2, k3, k5, k7}
\left\{ \begin{split}
  k_2 & <  \bar{\lambda}_1^2, \\
  \bar{\lambda}_3 & > \frac{2 \bar{\lambda}_1 k_2}{\left( \bar{\lambda}_1^2 - k_2 \right)},
 \end{split}  \right.
\end{equation}
one can verify from \eqref{eq:conditions for k2, k3, k5, k7} that $\kappa_3$, $\kappa_5$ and $\kappa_7$ are positive.

\section{Results and discussion}
\label{sec:results}

\begin{figure}[!t]
\centering
\subfigure[Tracking(Lyapunov)]{
\includegraphics[width=0.225\textwidth]{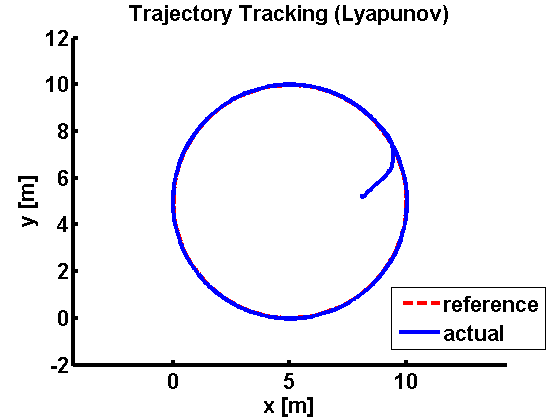}
\label{fig:circle tracking lyapunov}
}
\subfigure[Error(Lyapunov)]{
\includegraphics[width=0.225\textwidth]{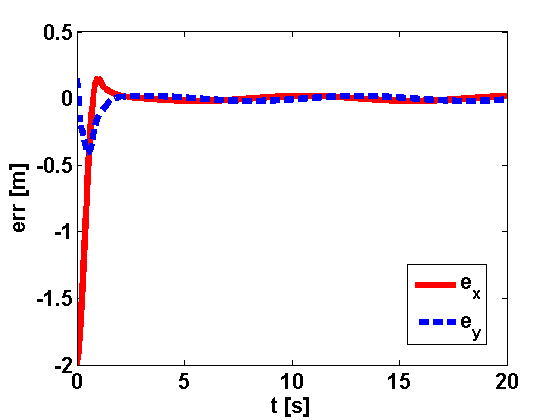}
\label{fig:error for circle tracking lyapunov}
}
\\
\subfigure[Tracking(Feedback linearization)]{
\includegraphics[width=0.225\textwidth]{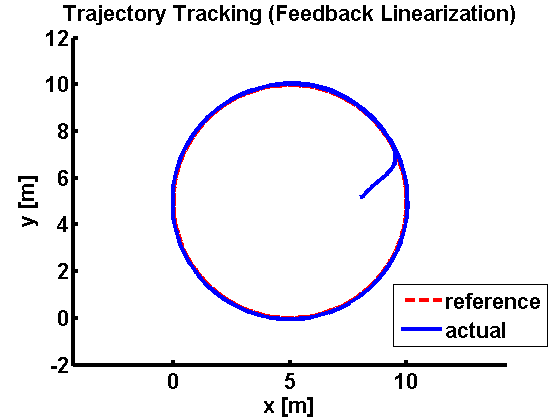}
\label{fig:circle tracking feedback}
}
\subfigure[Error(Feedback linearization)]{
\includegraphics[width=0.225\textwidth]{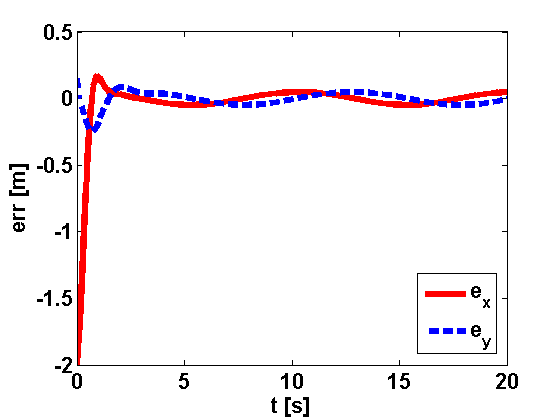}
\label{fig:error for circle tracking feedback}
}
\vspace{-.1truein}
\caption[]{For a reasonable comparison of the performance of the controller proposed in the present work and the dynamic feedback linearization controller proposed in \cite{caracciolo.1999.icra},
the gains are independently tuned for a circular trajectory with $5$ m radius in order to achieve similar behaviors in terms of the rising time, maximum peak and the decay ratio,
for both the error along the X- and Y- directions as shown in \subref{fig:error for circle tracking lyapunov} and \subref{fig:error for circle tracking feedback}.
}
\label{fig:gain tuning with circular trajectory}
\end{figure}

In this section, the performance of the controller proposed in the present work using a backstepping procedure that guarantees the Lyapunov stability
is compared to that of the controller proposed by \textit{Caracciolo et al.} in \cite{caracciolo.1999.icra} using the dynamic feedback linearization approach.
The comparison is performed using the MATLAB/SIMULINK.
The system \eqref{eq:equations of motion} is solved using MATLAB ode-solver of type \textit{ode5 (Dormand-Prince)} with a fixed time step (5ms).

In the first place, the considered initial conditions are $x_o = 8$ m, $y_o = 5$ m, $\theta_o = \pi/2$ rad, $\dot{x}_o = 0.5$ m/s, $\dot{y}_o = 0.5$ m/s, $\dot{\theta}_o = 0.1$ rad/s.
Next, the considered robot dimensions correspond to those of an ATRV-2 mobile robot used in \cite{caracciolo.1999.icra} with
$m = 116$ kg, $I=20$  $\text{kgm}^2$,  $a=0.37$ m, $b = 0.55$ m, $t =0.315$ m, $d_0 = 0.18$ m, and $r = 0.2$ m.

For a reasonable comparison between the two controllers, the gains are independently tuned for tracking a circular trajectory of $5$ m radius.
The criteria for choosing the gains for each controller are such that similar raising time, maximum peak and decay ratio are obtained for both cases while tracking the considered trajectory.
For the controller proposed in the present work, the conditions (\ref{eq:conditions for k1, k4, k6}) and (\ref{eq:conditions for k2, k3, k5, k7}) given in
Section \ref{subsec: gain tuning} must be also satisfied.
The resulting gains for the dynamic-feedback-linearization-based controller are
$k_{v_1} = 131$, $k_{a_1} = 20$, $k_{p_1} = 325$, $k_{v_2} = 210$, $k_{a_1} = 67$, and $k_{p_1} = 580$.
Whereas, for the Lyapunov-based controller, the resulting gains are
$k_1 = 3$, $k_2 = 15.8$, $\kappa_3 = 7.95$, $k_4 = 1$, $\kappa_5=0.0005$, $k_6=5$, and $\kappa_7=4.05$.

\begin{figure}[!t]
\centering
\subfigure[Tracking(Lyapunov)]{
\includegraphics[width=0.225\textwidth]{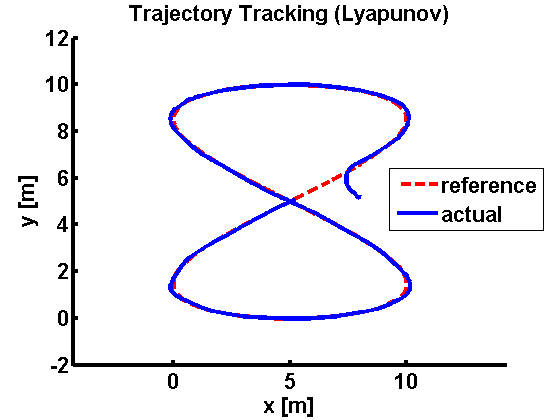}
\label{fig:tracking with noise lyapunov}
}
\subfigure[Error(Lyapunov)]{
\includegraphics[width=0.225\textwidth]{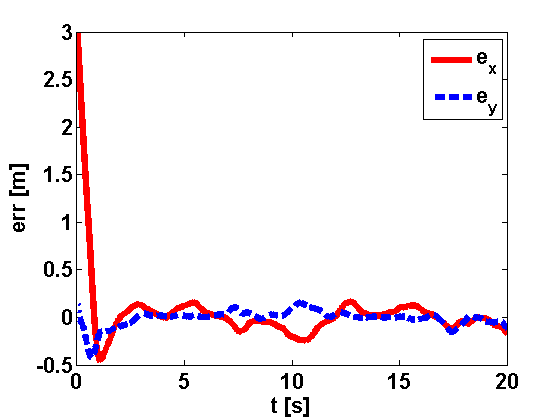}
\label{fig:error with noise lyapunov}
}
\\
\subfigure[Tracking(Feedback linearization)]{
\includegraphics[width=0.225\textwidth]{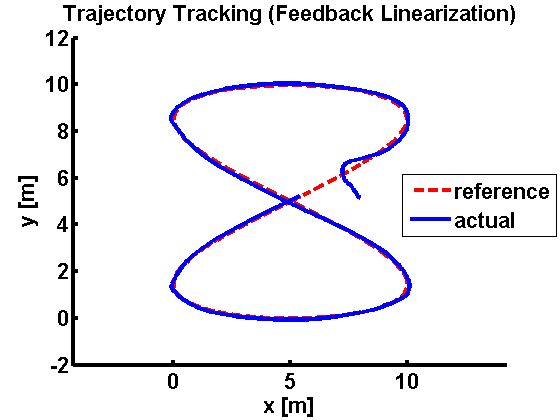}
\label{fig:tracking with noise feedback}
}
\subfigure[Error(Feedback linearization)]{
\includegraphics[width=0.225\textwidth]{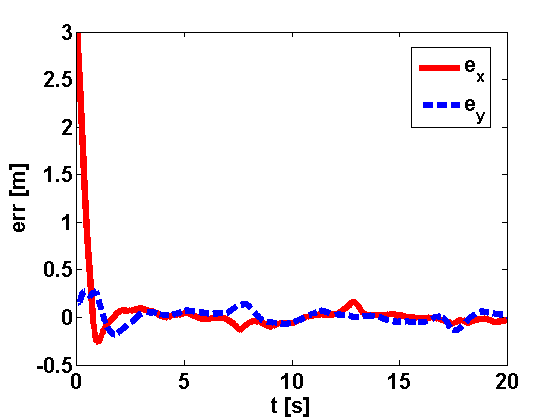}
\label{fig:error with noise feedback}
}
\vspace{-.1truein}
\caption[]{The performance of the proposed controller is compared to the dynamic feedback linearization approach proposed in \cite{caracciolo.1999.icra} while
the robot model is asked to track an eight-shaped Lissajous curve trajectory defined in (\ref{eq:lissajous trajectory}). In addition, a white Gaussian noise is added to the state vector
to emulate the sensor noise and to show the robustness of the controllers.
}
\label{fig:lissajous tracking with noise}
\end{figure}

In effect, Fig. \ref{fig:gain tuning with circular trajectory} shows the choice of such gains makes both controllers track the circular trajectory in a similar fashion.
However, in both cases the errors do not converge to zero but they oscillate.
Because both controllers are designed for the reduced dynamical system defined in (\ref{eq:explicit form of EoMs}), the error asymptotic convergence occurs for this system.
Whereas, when these control laws are used in the full dynamical system defined in (\ref{eq:equations of motion}), oscillatory behaviors can be observed from the results, and
this phenomenon might be due to the discrepancy that exists between the desired longitudinal component of the instantanous center of rotation, $d_0$,
imposed by the operational nonholonomic constraint defined in (\ref{eq:operational nonholonomic constraint})
and the actual ICR along the robot's longitudinal axis, $x_{\text{ICR}}$, as the robot tracks the desired trajectory.

Next, these gains are used to compare the performance of the two controllers in tracking an eight-shaped Lissajous-curve trajectory (shown in Fig. \ref{fig:lissajous tracking with noise}),
a curve characterized by its curvature that continuously changes. The considered Lissajous curve has the following expression
\begin{equation} \label{eq:lissajous trajectory}
 \boldsymbol{\xi}_r =
 \left\{
 \begin{aligned}
    & 5 \left( 1 + \sin{\left( \sqrt{0.4} \, t \right)} \right), \\
    & 5 \left( 1 + \sin{\left( \sqrt{0.4} \, t / 2 \right)} \right).
 \end{aligned}
 \right.
\end{equation}

Further, a multi-variate  white Gaussian noise is added to the state vector to emulate the sensor noise and study the robustness of both controllers.
The considered noise has the following mean and standard deviation values:
$\mu_x = 0$ m, $\mu_y = 0$ m, $\mu_\theta = 0$ rad, $\mu_{\dot{x}} = 0$ m/s, $\mu_{\dot{y}} = 0$ m/s, $\mu_{\dot{\theta}} = 0$ rad/s,
$\sigma_x = 0.02$ m, $\sigma_y = 0.02$ m, $\sigma_\theta = 0.01$ rad, $\sigma_{\dot{x}} = 0.08$ m/s, $\sigma_{\dot{y}} = 0.08$ m/s, and $\sigma_{\dot{\theta}} = 0.01$ rad/s.

\begin{figure}[!t]
\centering
\subfigure[Tracking(Lyapunov)]{
\includegraphics[width=0.225\textwidth]{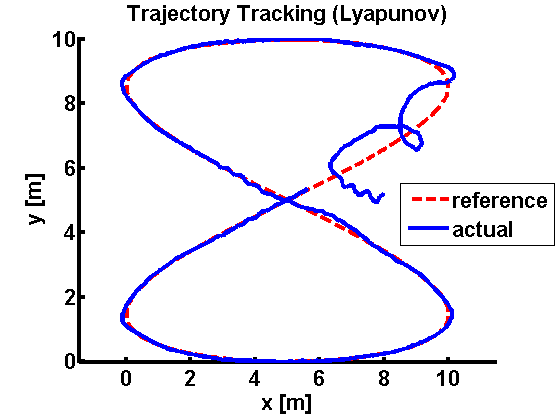}
\label{fig:tracking delay lyapunov}
}
\subfigure[Error(Lyapunov)]{
\includegraphics[width=0.225\textwidth]{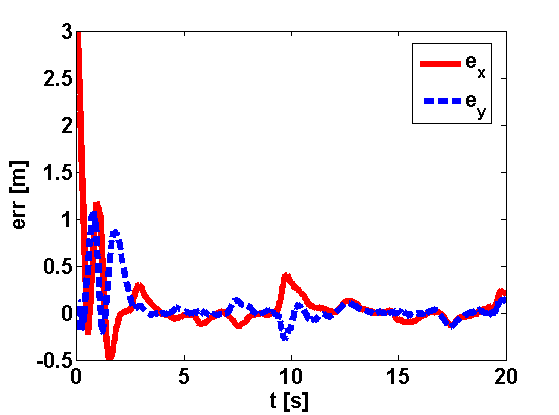}
\label{fig:error delay lyapunov}
}
\\
\subfigure[Tracking(Feedback linearization)]{
\includegraphics[width=0.225\textwidth]{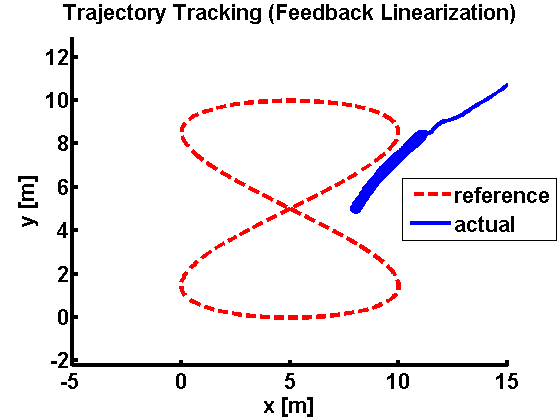}
\label{fig:tracking delay feedback}
}
\subfigure[Error(Feedback linearization)]{
\includegraphics[width=0.225\textwidth]{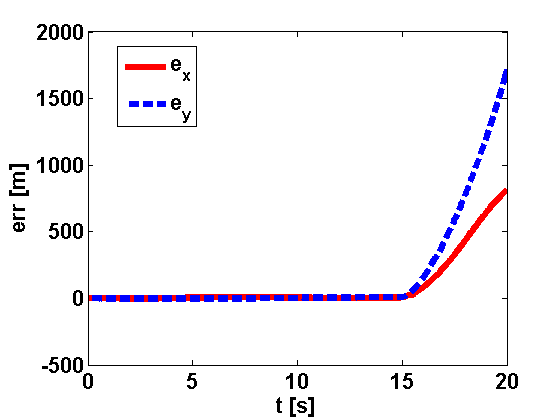}
\label{fig:error delay feedback}
}
\vspace{-.1truein}
\caption[]{The trajectory tracking by the proposed controller (\subref{fig:tracking delay lyapunov} and \subref{fig:error delay lyapunov}) is compared to the trajectory tracking
           by the controller proposed in \cite{caracciolo.1999.icra} (\subref{fig:tracking delay feedback} and \subref{fig:error delay feedback}) with the emulated sensory noise and
           a control time delay of $10$ ms along with a zero-order-holder to hold the control input signal for $10$ ms.
}
\label{fig:lissajous tracking with delay}
\end{figure}

Fig. \ref{fig:lissajous tracking with noise} shows the results of tracking the Lissajous curve trajectory with emulated sensory noise for both controllers.
In both cases, the controllers are able to track the reference trajectory even in the presence of the described noise. Notice that the error is accentuated when
tracking the four corners of the Lissajous curve, where the curvature abruptly changes. 

Finally, on top of the additive noise, a control time delay is also considered to further study the robustness of the system controlled by each of the considered controllers.
A control time delay of $10$ ms is introduced along with a zero-order-holder to hold the control input signal for $10$ ms.
The results shown in Fig. \ref{fig:tracking delay lyapunov} and Fig. \ref{fig:error delay lyapunov} reveal that the controller proposed in the present work
is able to track the desired trajectory, whereas this was not the case for the dynamic feedback linearization approach,
as the controller was unable to track the desired trajectory (see Fig. \ref{fig:tracking delay feedback} and Fig. \ref{fig:error delay feedback}).

\section{Conclusion and Future work}
\label{sec:conclusion}
In the present work, we propose a novel trajectory controller for a four-wheel skid-steering mobile robot using a backstepping technique guaranteeing the Lyapunov stability
on top of the dynamic model that \textit{Caracciolo et al.} proposed in \cite{caracciolo.1999.icra}.
Their feedback-linearization-based controller requires the acceleration state as well as the non-zero velocity constraint at any instant of time,
whereas the proposed controller does not require none of these preconditions.
Moreover, the proposed controller is robust in tracking trajectories even in the presence of measurement noise and control time delay.

In the near future, we will experimentally validate the performance of the proposed controller.
On the other hand, the error dynamics observed from using both controllers show that the error does not asymptotically vanish.
We believe that this effect is observed because the equality operational nonholonomic constraint used in the present work
overconstrains the instantaneous center of rotation to be at a fixed distance from the robot's center of gravity along the longitudinal direction.
In the future, we will relax this equality constraint into an inequality constraint with the hope to show asymptotic convergence.

\section*{Acknowledgement}
This work is partially supported by the RAPID-FRAUDO project (Num. 112906242) funded by the DGA (French Defence Agency).





\end{document}